\documentclass[12pt, a4paper]{article}
\usepackage[utf8]{inputenc}

\usepackage{graphicx}
\usepackage{lscape}
\usepackage{amsfonts}
\usepackage{amssymb,amsthm,amsmath,color,mathrsfs}
\usepackage[T1]{fontenc}
\usepackage{lmodern}
\usepackage{hyperref}
\usepackage{tikz-cd}
\usepackage{mathtools}

\usepackage{amsmath,amssymb}

\newtheorem{Th}{Theorem}[section]

\newtheorem{Prop}[Th]{Proposition}

\theoremstyle{definition}
\newtheorem{Def}[Th]{Definition}

\newtheorem{Rem}[Th]{Remark}
\newtheorem{Ex}[Th]{Example}

\theoremstyle{plain}

\let\ta\tau 
\renewcommand{\ta}{\scalebox{1.44}{$\tau$}}

\newcommand{\N}{\mathbb N}

\newcommand{\K}{\mathbb K}

\newcommand{\X}{\mathcal X}

\newcommand{\norm}{\mathcal{N}(\X)}

\title{Characterisation of equivalent norms on a linear space using exponential vector space}

\author{Dhruba Prakash Biswas\footnote{Department of Pure Mathematics, University of Calcutta, 35, Ballygunge Circular Road, Kolkata-700019, INDIA, e-mail : dhrubaprakash28@gmail.com}, Priti Sharma\footnote{Bangabasi College, 19, Rajkumar Chakraborty Sarani, Kolkata - 700009, India, e-mail : mspriti23@gmail.com} and Sandip Jana\footnote{Department of Pure Mathematics, University of Calcutta, 35, Ballygunge 
		Circular Road, Kolkata-700019, INDIA, 
		e-mail : sjpm@caluniv.ac.in}} 

\date{}

\begin{document}
	
	\maketitle
	\begin{abstract}
		
		In this paper we have found a necessary and sufficient condition for equivalence of two norms on a linear space using the theory of exponential vector space. Exponential vector space is an ordered algebraic structure which can be considered as an algebraic ordered extension of vector space. This structure is axiomatised on the basis of the intrinsic properties of the hyperspace $\mathscr{C}(\X)$ comprising all nonempty compact subsets of a Hausdorff topological vector space $\X$. Exponential vector space is a conglomeration of a semigroup structure, a scalar multiplication and a compatible partial order. We have shown that the collection of all norms defined on a linear space, together with the constant function zero, forms a topological exponential vector space. Then using the concept of comparing function (a concept defined on a topological exponential vector space) we have proved the aforesaid necessary and sufficient condition; also we have proved using comparing function that in an infinite dimensional linear space there are uncountably many non-equivalent norms.
	\end{abstract}
	
	
	
\section{Introduction}

It is a well-known fact in functional analysis that in a finite dimensional linear space all norms are equivalent. Again, in an infinite dimensional linear space uncountably many non-equivalent norms can be defined. In the present paper we shall prove a necessary and sufficient condition for equivalence of two norms on a linear space using the theory of exponential vector space. Exponential vector space is an ordered algebraic structure which can be considered as an algebraic ordered extension of vector space. This structure is axiomatised on the basis of the intrinsic properties of the hyperspace $\mathscr{C}(\X)$ comprising all nonempty compact subsets of a Hausdorff topological vector space $\X$. Exponential vector space is a conglomeration of a semigroup structure, a scalar multiplication and a compatible partial order. This structure was formulated by S. Ganguly et al. in \cite{C(X)} with the name `quasi-vector space'. Later Priti Sharma et al. study the same space in \cite{evs} with the new nomenclature `exponential vector space' considering various intrinsic properties of the space. Before going to other details let us present first the definition of exponential vector space.
	
	\begin{Def}\cite{evs}
		Let $(X,\leq)$ be a partially ordered set, `$+$' be a binary operation on 
		$X$ [called \emph{addition}] and `$\cdot$'$:K\times X\longrightarrow X$ be another composition [called \emph{scalar multiplication}, $K$ being a field]. If the operations and the partial order satisfy the following axioms then $(X,+,\cdot,\leq)$ is called an \emph{exponential vector space} (in short \emph{evs}) over $K$ [This structure was initiated with the terminology `\textit{quasi-vector spce}' or `\textit{qvs}' in short by S. Ganguly et al. in \cite{C(X)} ].
		
		$A_1:$ $(X,+)\text{ is a commutative semigroup with identity } \theta$

		$A_2:$ $x\leq y\, (x,y\in X)\Rightarrow x+z\leq y+z  \text{ and } \alpha\cdot 
		x\leq
		\alpha\cdot y, \forall\, z\in X,  \forall\, \alpha\in K$

		$A_3:$ (i) $\alpha\cdot(x+y)=\alpha\cdot x+\alpha\cdot y$

		\hspace*{0.80cm}(ii)  $\alpha\cdot(\beta\cdot x)=(\alpha\beta)\cdot x$

		\hspace*{0.80cm}(iii) $(\alpha+\beta)\cdot x \leq \alpha\cdot x+\beta \cdot x$

		\hspace*{0.80cm}(iv) $1\cdot x=x,\text{ where `1' is the multiplicative identity in }K$,

		\hspace*{0.30cm} $\forall\,x,y\in X,\  \forall\,\alpha,\beta\in K$

		$A_4:$ $\alpha\cdot x=\theta \text{ iff }\alpha=0\text{ or }x=\theta$

		$A_5:$ $x+(-1)\cdot x=\theta\text{ iff } x\in X_0:=\big\{z\in X:\ y\not\leq z, \forall\,y \in X\smallsetminus\{z\}\big \}$

		$A_6:$ For each $x \in X, \exists\,p\in X_0\text{ such that }p\leq x$. \label{d:evs}
	\end{Def}
	
	In the above definition, $X_0$ is precisely the set of all minimal elements of the evs $X$ with respect to the partial order on $X$ and it forms the maximum vector space (within $X$) over the same field as that of $X$ (\cite{C(X)}). This vector space $X_0$ is called the `\emph{primitive space}' or `\emph{zero space}' of $X$ and the elements of $X_0$ are called the `\emph{primitive elements}' [\cite{evs}]. 
	
	Thus every evs contains a vector space and conversely, given any vector space $V$, an evs $X$ can be constructed such that $V$ is isomorphic to $X_0$ \cite{evs}. In this sense,``exponential vector space'' can be considered as an algebraic ordered extension of vector space. The axiom $ A_3 $(iii) expresses very rapid growth of the non-primitive elements, since $x\leq \frac {1}{2} x +\frac{1}{2}x$, $\forall x \not \in \thinspace X_0$; whereas axiom $ A_6 $ demonstrates `\textit{positivity}' of all elements with respect to primitive elements. This justifies the nomenclature `exponential vector space'.
	
	\begin{Def}{\cite{Nach}}
		Let `$\leq$' be a preorder in a topological space $Z$; the preorder is said to be \textit{closed} if its graph $G_{\leq}(Z):=\{(x,y)\in Z\times Z: x\leq y\}$ is closed in $Z\times Z$ (endowed with the product topology).
	\end{Def}
	
	\begin{Th}{\em{\cite{Nach}}}
		A partial order `$\leq$' in a topological space $Z$ will be a closed order iff for any $x,y\in Z$ with $x\not\leq y$, $\exists$ open neighbourhoods $U,V$ of $x,y$ respectively in $Z$ such that $(\uparrow U)\cap(\downarrow V)=\emptyset$, where $\uparrow U:=\{x\in Z: x\geq u\text{ for some u }\in U\}$ and $\downarrow V:=\{x\in Z: x\leq v\text{ for some }v\in V\}$.\label{t:partcl}
	\end{Th}
	
	\begin{Def}\cite{evs}
		An exponential vector space $X$ over the field $\K$ of real or complex numbers is said to be a \textit{topological exponential vector space} if there exists a topology on $X$ with respect to which the addition, scalar multiplication are continuous and the partial order `$\leq$' becomes closed (Here $\K$ is equipped with the usual topology).
	\end{Def}
	\begin{Rem}
		If $X$ is a topological exponential vector space then its primitive space $X_0$ becomes a topological vector space, since restriction of a continuous function is continuous. Moreover, the closedness of the partial order `$\leq$' in a topological exponential vector space $X$ readily implies (in view of Theorem \ref{t:partcl}) that $X$ is Hausdorff and hence $X_0$ becomes a Hausdorff topological vector space.\label{rm:topvec}
	\end{Rem}

	\begin{Ex} \cite{mor}
		Let $X:=$ [$0,\infty$)$\times V$, where $V$ is a vector space over the field $\mathbb K$ of real or complex numbers. Define operations and partial order on $X$ as follows :
		for $(r,a),(s,b) \in X $ and $\alpha \in \mathbb K $,\\
		(i) $(r,a)+(s,b) := (r+s,a+b)$\\
		(ii) $\alpha (r,a) := (\lvert\alpha\rvert
		r,\alpha a)$\\
		(iii) $(r,a)\leq (s,b)$ iff $r\leq s$ and $a = b$\\
		Then [$0,\infty$)$\times V$ becomes an exponential vector space with the primitive space \{$0$\}$\times V$ which is clearly isomorphic to $V$.
		
		In this example, if we consider $V$ as a Hausdorff topological vector space, then [$0,\infty$)$\times V$ becomes a topological exponential vector space with respect to the product topology, where [$0,\infty$) is equipped with the subspace topology inherited from the real line $\mathbb R$.
		
		Instead of $V$, if we take the trivial vector space $ \{\theta\} $ in the above example, then the resulting topological evs is [$0,\infty$)$\times\{ \theta\}$ which can be clearly identified with the half ray [$0,\infty$) of the real line. Thus, [$0,\infty$) forms a topological evs over the field $\mathbb K$.\qed\label{e:vect}
	\end{Ex}
	
	\begin{Ex}{\cite{C(X)}}
		Let $\mathscr{C}(\X)$ denote the topological hyperspace consisting of all non-empty compact subsets of a Hausdorff topological vector space $\X$ over the field $\K$ of real or complex numbers. Define addition, scalar multiplication and partial order on $\mathscr{C}(\X)$ as follows:
		
		(i) For $A,B\in\mathscr{C}(\X)$, $A+B:=$ $\big\{a+b:$ $a\in A$, $b\in B\big\}$
		
		(ii) For $A\in\mathscr{C}(\X)$ and $\alpha\in\K$, $\alpha A:=$ $\big\{\alpha a:$ $a\in A\big\}$
		
		(iii) For $A,B\in\mathscr{C}(\X)$, $A\leq B\iff$ $A\subseteq B$
		
		Then $\mathscr{C}(\X)$ becomes an evs over the field $\K$. The primitive space is given by  $[\mathscr{C}(\X)]_{0}$ $=\big\{\{x\}:$ $x\in \X\big\}$.	Moreover, $\mathscr{C}(\X)$ forms a topological evs with respect to the Vietoris topology \cite{Mi}. An arbitrary basic open set in this topology is of the form $V_0^{+}\cap V_{1}^{-}\cap V_{2}^{-}\cap...\cap V_{m}^{-}$, where  $V_0,V_1,...,V_m$ are open in $\X$ with $V_i\subseteq V_0$ for all $i=1,2,...,m$. Here $ V^+:=\big\{A\in\mathscr{C}(\X):A\subseteq V\big\} $, $ V^-:=\big\{A\in\mathscr{C}(\X):A\cap V\neq\emptyset\big\} $, for any $ V\subseteq\X $.\qed\label{ex:vietor}
	\end{Ex}
	
	The study of exponential vector space (previously, quasi-vector space) was motivated by the inspection of properties of $\mathscr{C}(\X)$. It leaves a major influence in searching certain algebraic systems possessing the evs structure and developing general and specialized theory preserving the sense of hyperspace on exponential vector spaces.
	
	In the present paper we first show that the collection $\norm$ of all norms on a linear space $\X$ over the field $\K$, together with the constant function zero `$O$', forms a topological exponential vector space with respect to suitably defined operations, partial order and point open topology. Then using the concept of comparing function (defined on a topological exponential vector space in \cite{Pri}) we have proved some necessary and sufficient conditions of equivalence of two norms on any linear space. Also using the concept of comparing function we have shown that in an infinite dimensional linear space there are uncountably many non-equivalent norms (a standard result in functional analysis).
	
\section{Prerequisites}

	\begin{Def}\cite{mor}
		A map $\phi:X\to Y$ ($X,Y$ being two exponential vector spaces over the field $K$) is called an \textit{order-morphism} if
		
		(i) $f(x+y)=f(x)+f(y)$, $\forall\, x,y\in X$
		
		(ii) $f(\alpha x)=\alpha f(x)$, $\forall\alpha\in K$, $\forall\, x\in X$
		
		(iii) $x\leq y\ (x,y\in X)$ $\implies f(x)\leq f(y)$
		
		(iv) $p\leq q\ \big(p,q\in f(X)\big)$ $\implies f^{-1}(p)\subseteq\downarrow f^{-1}(q)$ and $f^{-1}(q)\subseteq\uparrow f^{-1}(p)$.
		\label{d:mor}\end{Def}
	
	Further, if $f$ is bijective (injective, surjective) order-morphism, then it is called \textit{order-isomorphism} (\textit{order-monomorphism}, \textit{order-epimorphism} respectively).
	
	If $X$ and $Y$ both are topological evs over the field $\K$, then the order-isomorphism $f:X\to Y$ is called \textit{topological order-isomorphism} if $f$ is a homeomorphism.
	
	\begin{Def}\cite{evs}
		A property of an evs is called \textit{evs property} if it remains invariant under order-isomorphism.
	\end{Def}
	
	\begin{Def}{\cite{JTh}}
		In an evs $X$, the \textit{primitive} of $x\in X$ is defined as the set\\
		\centerline{$P_{x}:=$ $\big\{p\in X_{0}:$ $p\leq x\big\}$}
	\end{Def}
	The axiom $A_6$ of the Definition \ref{d:evs} of an evs ascertains that $P_{x}\neq\emptyset$ for each $x\in X$. The elements of $P_{x}$ are known as \textit{primitive elements} of $x$.
	
	\begin{Def}{\cite{spri}}
		An evs $X$ is said to be a \textit{single primitive evs} if $P_{x}$ is a singleton set for each $x\in X$.
		\end{Def}
		
	\begin{Def}\cite{JTh}
		An evs $X$ is said to be a \textit{zero primitive} evs if $P_{x}=\{\theta\}$, for all $x\in X$.
		
		Clearly an evs $X$ is zero primitive iff $X_0=\{\theta\}$.
		Obviously, any zero primitive evs is a single primitive evs.
	\end{Def}
	
\begin{Def}\cite{JTh}
	 An element $x$ in an evs $X$ is said to be \textit{homogeneous} if $\alpha x =|\alpha|x$,
	$\forall\,\alpha\in\K$. An evs $X$ is said to be \textit{homogeneous} if each element of $X$ is homogeneous.
\end{Def}
	
\begin{Def}\cite{JTh}
	Let $X$ be an evs over the field $\K$ of real or complex numbers. An element
	$x\in X$ is said to be a \textit{convex element} if $(\alpha + \beta)x = \alpha x + \beta x, \forall\,\alpha,\beta\in\K$ with $\alpha,\beta\geq 0$. So each primitive element of $X$ is a convex
	element.
	
	An evs $X$ is said to be a \textit{convex} evs if every element of $X$ is convex.
\end{Def}	

\begin{Def}\cite{bal}
 Let $X$ be an evs over the field $\K$ of real or complex numbers. An element
$x\in X$ is said to be a \textit{balanced element} if $\alpha x \leq x,\forall\,\alpha\in\K$ with $|\alpha|\leq1$.

An evs $X$ is said to be \textit{balanced} if each element of $X$ is balanced.
\end{Def}

Single primitivity, convexity, homogeneity and balancedness are evs properties (\cite{spri},\cite{JTh},\cite{bal}).

\section{Evs structure on the collection of norms on a linear space}

Let $\X$ be a linear space over the field $\K$ of real or complex numbers and $\norm$ denote the set of all norms on $\X$ together with the constant function zero, say $O$, on $X$. We define addition, scalar multiplication and partial order on $\norm$ as follows :

(i) For $f,g\in\mathcal{N}(\X)$, $(f+g)(x):= f(x) + g(x)$ for all $x\in\X$.

(ii) For $f\in\mathcal{N}(\X)$ and for all $\alpha\in\K$, $(\alpha f)(x):= |\alpha| f(x)$ for all $x\in\X$.

(iii) For $f,g\in\mathcal{N}(\X)$, $f\leq g\iff f(x)\leq g(x)$ for all $x\in\X$. \\
We show below that  $\big(\mathcal{N}(\X),+,\cdot,\leq\big)$ becomes an evs over the field $\K$.

\begin{Th}
	$\norm$ is an exponential vector space over the field $\K$ of all real or complex numbers with respect to the aforesaid operations and partial order.
\label{t:norm}\end{Th}

\begin{proof}
We first show that if $f, g\in\norm$ with atleast one of $f,g$ being different from the zero function $O$, then $f+g$ is also a norm on $\X$ [Obviously, $O+O=O$].
\\Clearly, $(f+g)(x)=f(x)+g(x)\geq0$ and $(f+g)(x)=0 \Longleftrightarrow f(x)+g(x)=0 \Longleftrightarrow f(x)=0$ and $g(x)=0$ $\Longleftrightarrow$ $x=\theta$ (the zero element of $\X$) [ $\because$ atleast one of $f$ and $g$ is non-zero ].
\\ Again $(f+g)(x+y)=f(x+y)+g(x+y)\leq f(x)+f(y)+g(x)+g(y)=(f+g)(x)+(f+g)(y)$, $\forall x, y\in\X$.
\\ Now $(f+g)(\lambda x)=f(\lambda x)+g(\lambda x)=|\lambda| f(x)+|\lambda| g(x)=|\lambda|\big(f(x)+g(x)\big)=|\lambda|(f+g)(x)$, $\forall x\in\X$, $\forall \lambda\in\K$.
\\ Hence $f+g\in\norm$, $\forall f,g\in\norm$.

Now we show that for any $\alpha\in\K$ and $f\in\norm$, the scalar multiplication $\alpha f\in\norm$.
Clearly, $(\alpha f)(x)=|\alpha| f(x)\geq0$, $\forall x\in\X$. If $\alpha=0$ or $f=O$ then $\alpha f=O\in\norm$. For $\alpha\neq0$ and $f\neq O$, $(\alpha f)(x)=0\Longleftrightarrow|\alpha| f(x)=0\Longleftrightarrow f(x)=0\Longleftrightarrow x=\theta$.\\
 Now $(\alpha f)(x+y)=|\alpha| f(x+y)\leq|\alpha|\big(f(x)+f(y)\big)=(\alpha f)(x)+(\alpha f)(y)$, $\forall x, y\in\X$.
\\ Again $(\alpha f)(\lambda x)=|\alpha| f(\lambda x)=|\alpha||\lambda| f(x)=|\lambda|(\alpha f)(x)$, $\forall x\in\X$, $\forall \lambda\in\K$.
\\ Hence $\alpha f\in \norm$, $\forall f\in\norm$, $\forall \alpha\in\K$.

 Clearly, `$\leq$' is a partial order on $\norm$.

We now prove that $\big(\norm, +, \cdot, \leq\big)$ forms an evs over the field $\K$ of real or complex numbers.
\\$\bf{A_1:}$ Clearly, $(\norm, +)$ is a commutative semigroup with identity $O$, the constant function zero.
\\$\bf{A_2:}$ Let $f,g\in\norm$ with $f\leq g$. Then $f(x)\leq g(x), \forall x\in\X$. For any $h\in\norm$, $(f+h)(x)=f(x)+h(x)\leq g(x)+h(x)=(g+h)(x), \forall x\in\X$ $\Longrightarrow$ $f+h\leq g+h$, $\forall\,h\in\norm$.
\\Also, for any $\alpha\in\K$, $(\alpha f)(x)=|\alpha| f(x)\leq |\alpha| g(x)=(\alpha g)(x), \forall x\in\X$ $\Longrightarrow$ $\alpha f\leq \alpha g$, $\forall \alpha\in\K$.
\\$\bf{A_3\ (i):}$ Let $f,g\in\norm$ and $\alpha\in\K$. Then, $\big(\alpha(f+g)\big)(x)$ = $|\alpha|\big(f(x)+g(x)\big)$ = $|\alpha| f(x)+|\alpha| g(x)=(\alpha f+\alpha g)(x), \forall x\in\X$ $\Longrightarrow$ $\alpha(f+g)=\alpha f+\alpha g$, $\forall \alpha\in\K$, $\forall f,g\in\norm$.
\\$\bf{(ii)}$ Let $f\in\norm$ and $\alpha, \beta\in\K$. Then, $\big(\alpha(\beta f)\big)(x)=|\alpha|(\beta f)(x)=|\alpha||\beta| f(x)$ = $|\alpha\beta| f(x)=\big((\alpha\beta)f\big)(x), \forall x\in\X$ $\Longrightarrow$ $\alpha(\beta f)=(\alpha\beta)f$, $\forall \alpha,\beta\in\K$, $\forall f\in\norm$.
\\$\bf{(iii)}$ Let $f\in\norm$ and $\alpha, \beta\in\K$. Then, $(\alpha+\beta)f(x)$ = $|\alpha+\beta| f(x)$ $\leq$ $(|\alpha|+|\beta|)f(x)$ = $|\alpha| f(x)+|\beta| f(x)$ = $(\alpha f+\beta f)(x)$, $\forall x\in\X$ $\Longrightarrow$ $(\alpha+\beta)f\leq \alpha f+\beta f$, $\forall \alpha,\beta\in\K$, $\forall f\in\norm$.
\\$\bf{(iv)}$ Clearly, $1.f=f$, $\forall f\in\norm$.
\\$\bf{A_4:}$ Let $\alpha\in\K$ and $f\in\norm$. Then, $\alpha f=O$ $\Longleftrightarrow$ $(\alpha f)(x)=O(x), \forall x\in\X$ $\Longleftrightarrow$ $|\alpha| f(x)=0, \forall x\in\X$ $\Longleftrightarrow$  $\alpha=0$ or $f(x)=0, \forall x\in\X$ $\Longleftrightarrow$  $\alpha=0$ or $f=O$.
\\$\bf{A_5:}$ For any $f\in\norm$, $f+(-1)f=O$ $\Longleftrightarrow$ $\big(f+(-1)f\big)(x)=O(x), \forall x\in\X$ $\Longleftrightarrow$ $f(x)+f(x)=0, \forall x\in\X$ $\Longleftrightarrow$ $f(x)=0, \forall x\in\X$ i.e. $f=O$. Since $f(x)\geq0,\ \forall\,x\in\X,\forall\,f\in\norm$ we have $[\norm]_0=\{O\}$. Thus $f+(-1)f=O$ $\Longleftrightarrow$ $f\in[\norm]_0=\{O\}$.
\\$\bf{A_6:}$ Clearly, for each $f\in\norm$, $O\leq f$ where, $O\in[\norm]_0$.

Hence $\big(\norm,,+,\cdot,\leq\big)$ forms an exponential vector space over $\K$.
\end{proof}

\begin{Th}
	$\norm$ is a single primitive,  convex, homogeneous and balanced  evs.
\label{t:conhombal}\end{Th}

\begin{proof}
	As $[\norm]_0=\{O\}$ so $\norm$ is zero primitive and hence single primitive evs. For $\alpha, \beta\in\K$ with $\alpha,\beta\geq0$, $\big((\alpha+\beta)f\big)(x)=|\alpha+\beta| f(x)$ = $(|\alpha|+|\beta|)f(x)$ = $|\alpha| f(x)+|\beta| f(x)=(\alpha f+\beta f)(x), \forall x\in\X$ $\Longrightarrow$ $(\alpha+\beta)f=\alpha f+\beta f$ for $\alpha, \beta\geq0,\forall\,f\in\norm$. Hence, $\norm$ is a convex evs. Also for any $\alpha\in\K$ and any $f\in\norm$, $(\alpha f)(x)=|\alpha|f(x),\forall\,x\in\X$ $\implies$ $\alpha f=|\alpha|f,\forall\,f\in\norm$. Hence $\norm$ is homogeneous.	
	 Again, for $\alpha\in\K$ with $|\alpha|\leq1$, $(\alpha f)(x)=|\alpha| f(x)\leq f(x), \forall x\in\X$ $\Longrightarrow$ $\alpha f \leq f$ for $|\alpha|\leq1$. Thus, $\norm$ is a balanced evs.
\end{proof}

We now give the point open topology on $\norm$. For any $x\in\X$ and any open set $U$ in $[0,\infty)$, let $W(x,U):=\{f\in\norm : f(x)\in U \}$. Then $\big\{W(x,U): x\in\X$, $U$ is open in $[0,\infty)\big\}$ forms a subbase for the point open topology say, $\tau$ on $\norm$. Also, a net $\{f_n\}_{n\in D}$ [$D$ being a directed set] in $\norm$ converges to some $f\in\norm$ iff $f_n(x)\to f(x), \forall\, x\in\X$. 

\begin{Th}
	Consider the evs $\norm$ endowed with the point open topology. Then,
	\\(1) The addition `$+$' $:\norm\times\norm\longrightarrow\norm$ is continuous.
	\\(2) The scalar multiplication `$\cdot$' $: \K\times\norm\longrightarrow\norm$ is continuous, where $\K$ is endowed with the usual topology.
	\\(3) The partial order `$\leq$' is closed.
	\\ Thus $\norm$ with the point open topology forms a topological evs over $\K$.
\label{t:topnorm}\end{Th}

\begin{proof}
	(1) Let $\{f_n\}_{n\in D}$ and $\{g_n\}_{n\in D}$ be two nets in $\norm$ [ $D$ being a directed set ] such that $f_n\to f$ and $g_n\to g$ in $\norm$. Then $f_n(x)\to f(x)$ and $g_n(x)\to g(x)$ in $[0,\infty)$, $\forall x\in\X$ $\Longrightarrow$ $f_n(x)+g_n(x)\to f(x)+g(x)$, $\forall x\in\X$ [ $\because$ `+' is continuous in $[0,\infty)$ ] i.e. $(f_n+g_n)(x)\to(f+g)(x)$, $\forall x\in\X$ $\Longrightarrow$ $(f_n+g_n)\to(f+g)$ in $\norm$. Hence the addition is continuous in $\norm$.
	
	(2) Let $\{f_n\}_{n\in D}$ be a net in $\norm$ and $\{\alpha_n\}_{n\in D}$ be a net in $\K$ [$D$ being a directed set] such that $f_n\to f$ in $\norm$ and $\alpha_n\to \alpha$ in $\K$. Then $f_n(x)\to f(x)$, $\forall x\in\X$. Now for any $x\in \X$,	
	\begin{align*}
		&\big|(\alpha_n f_n)(x)-(\alpha f)(x)\big|= \big||\alpha_n| f_n(x)-|\alpha| f(x)\big|\\
		&\hspace{4cm}=\big||\alpha_n| f_n(x)-|\alpha_n| f(x)+|\alpha_n| f(x)-|\alpha| f(x)\big|\\
		&\hspace{4cm} \leq|\alpha_n| \big| f_n(x)-f(x)\big| + | f(x)|\big||\alpha_n|-|\alpha|\big|\\
		&\hspace{4cm}\to0 \ [ \because |\alpha_n|\to|\alpha|, f_n(x)\to f(x), \forall x\in\X ].
	\end{align*}
	$\therefore\ \alpha_n f_n\to\alpha f$ in $\norm$. Hence the scalar multiplication is continuous in $\norm$.
	
	(3) Let $\{(f_n, g_n)\}_{n\in D}$ be a net in $\norm\times \norm$ [ $D$ being a directed set ] converging to some $(f,g)\in \norm\times \norm$ with $f_n\leq g_n$, $\forall\, n\in D$. So, $f_n\to f$,  $g_n\to g$ in $\norm$. Therefore, $f_n(x)\to f(x)$, $g_n(x)\to g(x)$, $\forall x\in\X$ and $f_n(x)\leq g_n(x)$, $\forall x\in\X$, $\forall\, n\in D$ $\Longrightarrow$ $f(x)\leq g(x)$, $\forall x\in\X$ [ $\because$ `$\leq$' is closed in the topological evs $[0, \infty)$ ] $\Longrightarrow$ $f\leq g$ in $\norm$.
	Hence the partial order is closed in $\norm$.
	\\Therefore, $\big(\norm, \tau\big)$ is a topological evs over $\K$. Again, since $\norm\subseteq$ $[0,\infty)^\X$, where $[0,\infty)$ is a Tychonoff space and $\tau$ is the point open topology, so $\big(\norm, \tau\big)$ is also a Tychonoff space.
\end{proof}

\section{Characterisation of equivalent norms on a linear space}

In this section we first discuss the concept of \textit{comparing function} defined on a zero primitive topological evs. Then we shall compute the comparing function on the zero primitive topological evs $\norm$. Finally, using this comparing function we shall prove necessary and sufficient condition for equivalence of two norms on a linear space over $\K$. We shall conclude this paper by proving, with the help of comparing function, a standard result from functional analysis that in an infinite dimensional linear space there are uncountably many non-equivalent norms.

\begin{Def}
	Let $X$ be a zero primitive evs over the field $\K$ of real or complex numbers with additive identity $\theta$. For $x,y\in X$, we define the \textit{comparing spectrum of $y$ relative to $x$} by $\sigma_{x}(y):=\big\{\lambda\in\K:$ $\lambda x\leq y\big\}$.
	
	Since $X$ is zero primitive, $y\geq\theta,\forall\,y\in X$ and hence $0\in\sigma_x(y)$ so that it is nonempty.
\end{Def}

\begin{Prop}{\em\cite{bal}}
For a topological evs $X$, $\sigma_x(y)$ is bounded, $\forall\,x\neq\theta,\forall\,y\in X$.
\label{p:specbdd}\end{Prop}

\begin{Def}{\cite{Pri}}
	Let $X$ be a zero primitive topological evs over the field $\K$ of real or complex numbers. For any $x\,(\neq\theta)\in X$, the \textit{comparing function relative to $x$}, denoted by $C_{x}$, is defined as\\ \centerline{$C_{x}(y):=\underset{\lambda\in\sigma_{x}(y)}{\sup}|\lambda|$, for all $y\in X$.}
	
	By Proposition \ref{p:specbdd}, $C_x$ is well-defined and $C_x(y)\in[0,\infty), \forall\,y\in X,\forall\,x\neq\theta$.
\end{Def}

\begin{Prop}
	Let $X$ be a zero primitive topological evs over the field $\K$. For any non-zero homogeneous element $x\in X$ and any $y\in X$, $C_x(y)x\leq y$.
\label{p:leq}\end{Prop}

\begin{proof}
	By definition of $C_x(y)$, $\exists$ a sequence $\{\lambda_n\}$ in $\sigma_x(y)$ such that $|\lambda_n|\to C_x(y)$ as $n\to\infty$. Then $\lambda_nx\leq y,\forall\,n\in\N$. Now $x$ being homogeneous, $|\lambda_n|x=\lambda_nx,\forall\,n$. So $|\lambda_n|x\leq y,\forall\,n$. Therefore the scalar multiplication being continuous and the partial order `$\leq$' being closed we have $C_x(y)x\leq y$.
\end{proof}

We shall now compute the comparing function in the topological evs $\norm$.

\begin{Th}
	Consider the topological evs $\mathcal{N}(\X)$ and let $f\,(\neq O)\in\mathcal{N}(\X)$. Then the comparing function relative to $f$ is $C_{f}(g)=\underset{x\in\X\smallsetminus\{\theta\}}{\inf}\frac{g(x)}{f(x)}$, $\forall\,g\in\mathcal{N}(\X)$. Here $\theta$ is the zero element in the vector space $\X$.
\end{Th}  

\begin{proof}
	Here $\sigma_{f}(g)$  $=\big\{\lambda\in\K:\lambda f\leq g\big\}=\big\{\lambda\in\K:|\lambda| f(x)\leq g(x),\text{for all }x\in\X\big\}$. Define $s:=$ $\underset{x\in\X\smallsetminus\{\theta\}}{\inf}$ $\frac{g(x)}{f(x)}$. Clearly $s\in[0,\infty)$. We will prove that $C_{f}(g)=s$.

	By definition of $s$, $s \leq$ $\frac{g(x)}{f(x)}$ for all $x\in\X\smallsetminus\{\theta\}$ $\implies s f(x)\leq g(x)$ for all $x\in\X$ $\implies s\cdot f\leq g$ $\implies s\in \sigma_{f}(g)$. Hence $s\leq\underset{\lambda\in\sigma_{f}(g)}{\sup}|\lambda|$ $= C_{f}(g)$. 
	
	To prove the converse, let $\lambda\in\sigma_{f}(g)$. Therefore $\lambda\cdot f\leq g$ $\implies |\lambda| f(x)\leq g(x)$ for all $x\in\X$ $\implies|\lambda|\leq$ $\frac{g(x)}{f(x)}$ for all $x\in\X\smallsetminus\{\theta\}$ $\implies$ $|\lambda|\leq$ $\underset{x\in\X\smallsetminus\{\theta\}}{\inf}$ $\frac{g(x)}{f(x)}$ $= s$. This holds for any $\lambda\in\sigma_{f}(g)$. Therefore $\underset{\lambda\in\sigma_{f}(g)}{\sup}|\lambda|$ $\leq s$ i.e. $C_{f}(g)\leq s$.

	Consequently, $C_{f}(g)$ $ =\underset{x\in\X\smallsetminus\{\theta\}}{\inf}$ $\frac{g(x)}{f(x)}$, for all $g\in\mathcal{N}(\X)$.	
\end{proof}

\begin{Th}
Let $f,g$ be two norms on a linear space $\X$. Then $g$ produces larger topology than the topology produced by  $f$ if and only if $C_f(g)\neq0$.
\end{Th}

\begin{proof}
Since $\norm$ is a homogeneous, zero primitive, topological evs (by Theorems \ref{t:conhombal} and \ref{t:topnorm}), we have by Proposition \ref{p:leq}, $C_f(g)f\leq g$ $\implies$ $C_f(g)f(x)\leq g(x),\forall\,x\in\X$. Let $\ta(f)$ and $\ta(g)$ be the topologies generated by the norms $f,g$ respectively. Then $C_f(g)\neq0\implies\ta(g)\supseteq\ta(f)$.

Conversely, let $\ta(g)\supseteq\ta(f)$. Then $\exists\,\lambda>0$ such that $\lambda f(x)\leq g(x),\forall\,x\in\X$ $\implies$ $\lambda f\leq g$ $\implies$ $\lambda\in\sigma_f(g)$ $\implies$ $\lambda\leq C_f(g)$ $\implies C_f(g)\neq0$.
\end{proof}	
\begin{Th}
	In a linear space $\X$ over the field $\K$ of all real or complex numbers, any two norms $f,g$ are equivalent if and only if $C_f(g)C_g(f)\neq0$.
\label{t:NASC}\end{Th}	

\begin{proof}
	Let $f,g$ be two equivalent norms on $\X$. Then there exists $\lambda,\mu>0$ such that $\lambda f(x)\leq g(x)\leq \mu f(x),\forall\,x\in\X$ $\implies$ $\lambda\in\sigma_f(g)$ and $\frac{1}{\mu}\in\sigma_g(f)$. So by definition, $\lambda\leq C_f(g)$ and $\frac{1}{\mu}\leq C_g(f)$ $\implies$ $C_f(g)C_g(f)\neq0$.
	
	Conversely let $C_f(g)C_g(f)\neq0$. Since $\norm$ is a homogeneous, zero primitive, topological evs (by Theorems \ref{t:conhombal} and \ref{t:topnorm}), we have by Proposition \ref{p:leq} that $C_f(g)f\leq g\leq\frac{1}{C_g(f)}f$ $\implies$ $C_f(g)f(x)\leq g(x)\leq\frac{1}{C_g(f)}f(x)$, for all $x\in\X$. This justifies that $f,g$ are two equivalent norms on $\X$.
\end{proof}

\begin{Ex}
	Let $c_{_{00}}$ be the linear space of all sequences in $\K$ whose all but finitely many terms are zero. Then the norms $\|\cdot\|_{\infty}$ and $\|\cdot\|_1$ on $c_{_{00}}$ are non-equivalent. To justify this we show that $C_{\|\cdot\|_1}(\|\cdot\|_{\infty})=0$. Let us consider a sequence $\{x^n\}$ in $c_{_{00}}$, where $x^n=(1,2,\ldots,n,0,0,\ldots)\in c_{_{00}}$. Then $C_{\|\cdot\|_1}(\|\cdot\|_{\infty})=\inf\limits_{\substack{x\in c_{_{00}}\\ x\neq0}}\frac{\|x\|_{\infty}}{\|x\|_{1}}\leq\inf\limits_{n\in\N}\frac{\|x^n\|_{\infty}}{\|x^n\|_{1}}=\inf\limits_{n\in\N}\frac{n}{1+2+\cdots+n}=\lim\limits_{n\to\infty}\frac{2}{n+1}=0 $.\qed
\end{Ex}

This example is instructive to show that there exists non-equivalent norms on any infinite dimensional linear space. We shall prove this standard result from functional analysis using comparing function.

\begin{Th}
	Let $\X$ be an infinite dimensional linear space over the field $\K$. Then there exist non-equivalent norms on $\X$.
\end{Th}

\begin{proof}
Let $H$ be a Hamel basis of $\X$.  Since $H$ is infinite we can find an enumerable subset $B:=\{e_n\in H:n\in\N\}$ of $H$. Now each $x\in\X$ can be expressed uniquely as $x=\sum\limits_{i=1}^n\alpha_ih_i$, for some $h_i\in H$ and $\alpha_i\in\K$. We define two norms on $\X$ as follows :\\
\centerline{$\|x\|_{\infty}:=\max\limits_{1\leq i\leq n}|\alpha_i|$ and $\|x\|_1:=\sum\limits_{i=1}^n|\alpha_i|$, if $x$ has the representation as above.}\\
We show below that these two norms are not equivalent.

Actually we show that $C_{\|\cdot\|_1}(\|\cdot\|_{\infty})=0$. For this we construct a sequence $\{x_n\}$ in $\X$ by $x_n:=e_1+2e_2+\cdots+ne_n,\forall\,n\in\N$. Then $\|x_n\|_{\infty}=n$ and $\|x_n\|_1=1+2+\cdots+n$. Therefore $C_{\|\cdot\|_1}(\|\cdot\|_{\infty})=\inf\limits_{\substack{x\in\X\\ x\neq0}}\frac{\|x\|_{\infty}}{\|x\|_{1}}\leq\inf\limits_{n\in\N}\frac{\|x_n\|_{\infty}}{\|x_n\|_{1}}=\inf\limits_{n\in\N}\frac{n}{1+2+\cdots+n}=\lim\limits_{n\to\infty}\frac{2}{n+1}=0 $. Then the theorem follows from Theorem \ref{t:NASC}.
\end{proof}

We now show that there are uncountably many non-equivalent norms on any infinite dimensional linear space using comparing functions.

\begin{Th}
	Let $\X$ be an infinite dimensional linear space over the field $\K$. Then there exist uncountably many non-equivalent norms on $\X$.
\end{Th}

\begin{proof}
	Let $H$ be a Hamel basis of $\X$.  Since $H$ is infinite we can find an enumerable subset $B:=\{e_n\in H:n\in\N\}$ of $H$. Now each $x\in\X$ can be expressed uniquely as $x=\sum\limits_{i=1}^n\alpha_ih_i$, for some $h_i\in H$ and $\alpha_i\in\K$. We define $p$-norm ($p\geq1$) on $\X$ as follows :\\
	\centerline{ $\|x\|_p:=\left(\sum\limits_{i=1}^n|\alpha_i|^p\right)^{\frac{1}{p}}$, if $x$ has the representation as above.}\\
	We show below that this gives non-equivalent norms for different $p$'s.

	We show that for $p>q$, $C_{\|\cdot\|_q}(\|\cdot\|_p)=0$. For this we construct a sequence $\{x_n\}$ in $\X$ by $x_n:=ne_1+ne_2+\cdots+ne_n,\forall\,n\in\N$. Then $\|x_n\|_p=n^{1+\frac{1}{p}}$ and $\|x_n\|_q=n^{1+\frac{1}{q}}$. Therefore $C_{\|\cdot\|_q}(\|\cdot\|_p)=\inf\limits_{\substack{x\in\X\\ x\neq0}}\frac{\|x\|_p}{\|x\|_q}\leq\inf\limits_{n\in\N}\frac{\|x_n\|_p}{\|x_n\|_q}=\inf\limits_{n\in\N}\frac{n^{1+\frac{1}{p}}}{n^{1+\frac{1}{q}}}=\lim\limits_{n\to\infty}n^{\frac{1}{p}-\frac{1}{q}}=0\ \big[\because\frac{1}{p}-\frac{1}{q}<0\big]$. Then the theorem follows from Theorem \ref{t:NASC}.
\end{proof}

Let $\mathcal{N_*}(\X):=\norm\smallsetminus\{O\}$. Define a map $\Psi:\mathcal{N_*}(\X)\times\mathcal{N_*}(\X)\to[0,\infty)$ by\\
\centerline{$\Psi(f,g):=\min\big\{C_f(g),C_g(f)\big\},\forall\,f,g\in\mathcal{N_*}(\X)$.}

Then we have the following characterisation of non-equivalent norms.

\begin{Th}
	Two norms $f,g$ on a linear space $\X$ over the field $\K$ are non-equivalent if and only if $\Psi(f,g)=0$.
\end{Th}

\begin{proof}
	This follows from the Theorem \ref{t:NASC}.
\end{proof}

\noindent\textbf{\underline{Acknowledgement} :} The first author is thankful to University Grants Commission, India for financial assistance.

\end{document}